\documentclass[11pt,reqno]{amsart}
\usepackage{amsmath,enumerate,etoolbox,booktabs,tikz,float}
\usepackage[section]{placeins}
\usepackage{txfonts}
\RequirePackage[scaled=0.92]{helvet}

\RequirePackage[paperwidth=155mm,paperheight=235mm,tmargin=20mm,lmargin=19mm,rmargin=19mm,bmargin=16mm,asymmetric]{geometry}
\patchcmd{\section}{\scshape}{\bfseries}{}{}
\makeatletter
\renewcommand{\@secnumfont}{\bfseries}
\renewcommand\@biblabel[1]{#1.}
\makeatother
\renewcommand{\chi}{\chiup}
\newtheorem{theorem}{Theorem}
\newtheorem*{theorem*}{Theorem}

\newtheorem*{proposition*}{Proposition}

\newtheorem*{lemma*}{Lemma}

\newtheorem*{corollary*}{Corollary}
\newtheorem{conjecture}{Conjecture}
\newtheorem*{conjecture*}{Conjecture}
\theoremstyle{definition}

\newtheorem*{question*}{Question}
\newtheorem*{definition*}{Definition}

\newtheorem*{remark*}{Remark}

\newtheorem*{example*}{Example}
\title[On parity and characters]{Note on parity and the irreducible characters
  of the symmetric group}
\author[A.~R.~Miller]{Alexander~R.~Miller}
\address{\raisebox{-20pt}[0pt][0pt]{Fakult\"at f\"ur Mathematik, Universit\"at Wien, Vienna, Austria}}
\thanks{The author was supported
  in part by the Austrian Science Foundation FWF Special Research Program
``Algorithmic and Enumerative Combinatorics'' (SFB F50).}
\begin{document}
\maketitle
\parindent1em
\section*{Introduction}\label{Introduction:Section}
\noindent
The object of this short note is to prove a theorem and present a conjecture
for the number of even entries in the character table of the symmetric group~$S_n$.
\begin{theorem}\label{Main:Theorem}
  The number of even entries in the character table of $S_n$ is even.
\end{theorem}
\begin{conjecture}\label{Main:Conjecture}
  The proportion of the character table of $S_n$ covered by even entries
  tends to $1$ as $n\to \infty$.
\end{conjecture}

Theorem~\ref{Main:Theorem} is proved in Section~\ref{Theorem:Section}.
Conjecture~\ref{Main:Conjecture} is discussed in Section~\ref{Conjectures:Section}.
To support Conjecture~\ref{Main:Conjecture} we write down
in Table~\ref{EO:Table} 
the number of even entries and odd entries in the character table of $S_n$ for $1\leq n\leq 76$. 
See Figure~\ref{Figure:Plot}.
Another table (Table~\ref{Other:Table}) in Section~\ref{Conjectures:Section}  
suggest a more general phenomenon.
\begin{conjecture}\label{General:Conjecture}
  The proportion of the character table of $S_n$ covered by entries divisible by
  a given prime number $p$ tends to $1$ as $n\to\infty$.
\end{conjecture}

\begin{figure}[h]
  \centering
  \begin{tikzpicture}[scale=4]
    \draw [thick] (0,0-.004)--(0,1)--(2+.01,1)--(2+.01,0-.004)--cycle;
    \foreach \i in {2,4,6,8} \draw [thick] (0,\i/10) -- (.03,\i/10);
    \foreach \i in {2,4,6,8} \node[rotate=90] at (0-.05,\i/10) {\tiny\bf 0.\i};
    \node[rotate=90] at (0-.05,1) {\tiny\bf 1.0};
    \node[rotate=90] at (0-.05,0-.002) {\tiny\bf 0.0};
    \foreach \i in {10,20,...,76} \draw [thick] (2*\i/76,0-.004)--(2*\i/76,.03-.004);
    \foreach \i in {0,10,...,76} \node at (2*\i/76,0-.004-.05) {\tiny\bf \i};
  \foreach \x/\y in
  {
1/ 0.000000,
2/ 0.000000,
3/ 0.222222,
4/ 0.240000,
5/ 0.326530,
6/ 0.363636,
7/ 0.400000,
8/ 0.549587,
9/ 0.564445,
10/ 0.547619,
11/ 0.581633,
12/ 0.598414,
13/ 0.597392,
14/ 0.635720,
15/ 0.621578,
16/ 0.702611,
17/ 0.701470,
18/ 0.695632,
19/ 0.711920,
20/ 0.729827,
21/ 0.727850,
22/ 0.756931,
23/ 0.755566,
24/ 0.761004,
25/ 0.766302,
26/ 0.776543,
27/ 0.773131,
28/ 0.791316,
29/ 0.785669,
30/ 0.791326,
31/ 0.790687,
32/ 0.808667,
33/ 0.803730,
34/ 0.810735,
35/ 0.811763,
36/ 0.815064,
37/ 0.815565,
38/ 0.824422,
39/ 0.822188,
40/ 0.827024,
41/ 0.827150,
42/ 0.832165,
43/ 0.830679,
44/ 0.837467,
45/ 0.835640,
46/ 0.839228,
47/ 0.839611,
48/ 0.844193,
49/ 0.843245,
50/ 0.847102,
51/ 0.847389,
52/ 0.849612,
53/ 0.850410,
54/ 0.853485,
55/ 0.853151,
56/ 0.855968,
57/ 0.856350,
58/ 0.858603,
59/ 0.858999,
60/ 0.860868,
61/ 0.860969,
62/ 0.862982,
63/ 0.863306,
64/ 0.865487,
65/ 0.865570,
66/ 0.866821,
67/ 0.867150,
68/ 0.869040,
69/ 0.869127,
70/ 0.870139,
71/ 0.870653,
72/ 0.871719,
73/ 0.871960,
74/ 0.873174,
75/ 0.873273,
76/ 0.874223
  }
  \draw [fill] ({(2*\x)/76}, \y) circle (0.005);
\end{tikzpicture}
\caption{Proportion
  of the character table of the symmetric group $S_n$
  covered by even entries for $1\leq n\leq 76$.}
\label{Figure:Plot}
\end{figure}
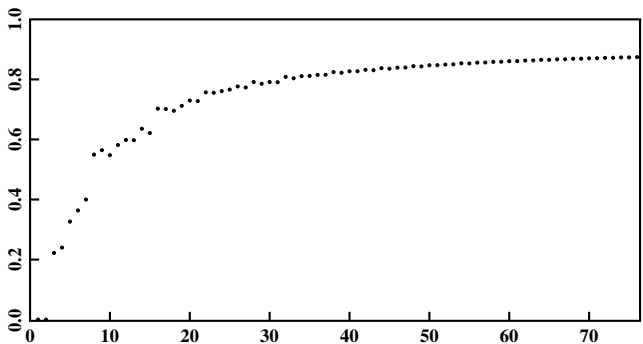

\section{Proof of Theorem~\ref{Main:Theorem}}\label{Theorem:Section}
\noindent
Let $p_n$ be the number of partitions of $n$.
Here a partition of $n$ is a sequence of 
positive integers $\lambda=(\lambda_1,\lambda_2,\ldots,\lambda_\ell)$
such that $\lambda_1\geq \lambda_2\geq \ldots\geq \lambda_\ell$ and
$\lambda_1+\lambda_2+\ldots+\lambda_\ell=n$.
The conjugate of $\lambda$ is the partition $\lambda'$ whose parts are 
$\lambda_i'=\#\{j : i\leq \lambda_j\}$ for $1\leq i\leq \lambda_1$.
Conjugation is the involution ${\lambda\mapsto \lambda'}$.
The fixed points of this involution are  self-conjugate partitions.
Self-conjugate partitions $\lambda$ of $n$ are in one-to-one correspondence
with  partitions $\mu$ of $n$ into odd distinct parts
via 
$\lambda\mapsto \mu$
where $\mu_i=2(\lambda_i-i)+1$ for $i$ such that $1\leq i\leq \lambda_i$. 
\begin{proof}[Proof of Theorem~\ref{Main:Theorem}]
  Let $O_n$ be the number of odd entries in the character table of $S_n$.
  Then
  \begin{equation}\label{Odd:Eq}
    O_n\equiv \sum_g\sum_\chi \chi(g)\equiv \sum_g\sum_\chi \chi(g)^2\pmod{2}
  \end{equation}
  where the two outer sums run over a set of representatives $g$ for the conjugacy classes
  and the two inner sums run over the irreducible characters $\chi$.
  Moreover one of the orthogonality relations \cite{F}  tells us that 
  \begin{equation}\label{Orth:Eq}
    \sum_{\chi}\chi(g)^2=1^{m_1}m_1!2^{m_2}m_2!\ldots n^{m_n}m_n!
  \end{equation}
  for $m_p$ the number of cycles of period $p$ in the cycle decomposition of~$g$.
  Together \eqref{Odd:Eq} and \eqref{Orth:Eq} imply
  \begin{equation}
    O_n\equiv OD_n \pmod{2}
  \end{equation}
  where $OD_n$ is the number of partitions of $n$ into odd distinct parts. 
  Let $SC_n$ be the number of self-conjugate partitions of $n$ so that 
  $SC_n=OD_n$ and hence
  \begin{equation}\label{O:SC:Eq}
    O_n\equiv SC_n\pmod{2}.
  \end{equation}
  
  Let $E_n$ be the number of even entries in the caracter table of $S_n$. Then 
  \begin{equation}\label{O:E:Eq}
    O_n+E_n=p_n^2\equiv p_n\pmod{2}.
  \end{equation}
  Together \eqref{O:SC:Eq} and \eqref{O:E:Eq} imply
  \begin{equation}
    E_n\equiv p_n-SC_n\pmod{2}.
  \end{equation}
  But $p_n-SC_n\equiv 0\pmod{2}$ because conjugation restricts to a fixed-point-free involution on the
  set of non-self-conjugate partitions of $n$.
\end{proof}

\section{Remarks and some tables}\label{Conjectures:Section}
\noindent
This section contains some tables and remarks.
The main object is Table~\ref{EO:Table}
for the number of even entries in the character table of $S_n$.

\subsection{Remarks}\label{Remarks:Section}
\noindent
Let $\chi(\mu)$ be short for the constant value $\chi(g)$
of the irreducible character $\chi$ of $S_n$  on the
class consisting of all permutations $g\in S_n$
for which the periods of the disjoint 
cycles form the partition $\mu$.
\subsubsection{}
In terms of the probability that 
an entry $\chi(\mu)$ is even  when chosen uniformly at random
from the character table of the symmetric group $S_n$ Conjecture~\ref{Main:Conjecture} says
\begin{equation}
  {\rm Prob}(\,\chi(\mu)\text{ is even}\,)\to 1\text{ as }n\to\infty.
\end{equation}
This parity bias becomes even more striking
when compared with the distribution of signs in the character table of $S_n$
(cf.~\cite[Question~3]{M}). See Figure~\ref{Figure:PM:Plot} and 
Table~\ref{PMZ:Table}.
\begin{conjecture}
  ${\rm Prob}(\,\chi(\mu)>0 \mid \chi(\mu)\neq 0\,)\to 1/2$ as $n\to\infty$.
\end{conjecture}
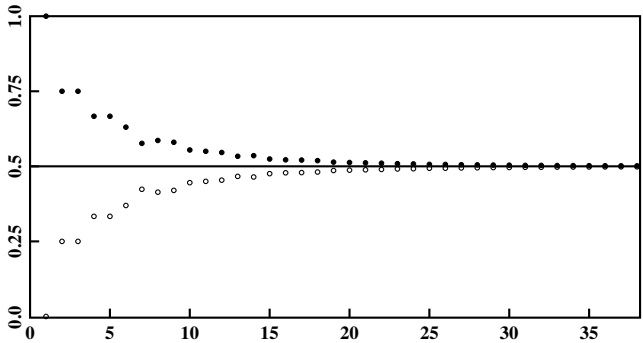
\begin{figure}[hbt]
  \centering
  \begin{tikzpicture}[scale=4]
    \draw [thick] (0,0-.004)--(0,1)--(2+.01,1)--(2+.01,0-.004)--cycle;
    \foreach \i in {0.25,0.5,0.75} \draw [thick] (0,\i) -- (.03,\i);
    \foreach \i in {0.25,0.5,0.75} \node[rotate=90] at (0-.05,\i) {\tiny\bf \i};
    \node[rotate=90] at (0-.05,1) {\tiny\bf 1.0};
    \node[rotate=90] at (0-.05,0-.002) {\tiny\bf 0.0};
    \foreach \i in {5,10,...,38} \draw [thick] (4*\i/76,0-.004)--(4*\i/76,.03-.004);
    \foreach \i in {0,5,...,38} \node at (4*\i/76,0-.004-.05) {\tiny\bf \i};
    \foreach \x/\y in
{
1 / 1.00000,
2 / 0.750000,
3 / 0.750000,
4 / 0.666667,
5 / 0.666667,
6 / 0.630435,
7 / 0.576470,
8 / 0.586102,
9 / 0.580101,
10 / 0.554421,
11 / 0.550047,
12 / 0.546094,
13 / 0.533494,
14 / 0.535560,
15 / 0.524393,
16 / 0.521590,
17 / 0.520837,
18 / 0.518988,
19 / 0.513955,
20 / 0.512761,
21 / 0.511580,
22 / 0.510413,
23 / 0.508808,
24 / 0.507993,
25 / 0.506464,
26 / 0.506354,
27 / 0.505301,
28 / 0.504786,
29 / 0.504102,
30 / 0.503858,
31 / 0.503163,
32 / 0.503088,
33 / 0.502617,
34 / 0.502368,
35 / 0.502168,
36 / 0.501929,
37 / 0.501712,
38 / 0.501601
}
\draw [fill] ({(4*\x)/76}, \y) circle (0.0075);
\foreach \x/\y in
{
1 / 0.000000,
2 / 0.250000,
3 / 0.250000,
4 / 0.333333,
5 / 0.333333,
6 / 0.369565,
7 / 0.423530,
8 / 0.413898,
9 / 0.419899,
10 / 0.445578,
11 / 0.449953,
12 / 0.453907,
13 / 0.466506,
14 / 0.464440,
15 / 0.475607,
16 / 0.478410,
17 / 0.479163,
18 / 0.481013,
19 / 0.486044,
20 / 0.487239,
21 / 0.488420,
22 / 0.489587,
23 / 0.491192,
24 / 0.492007,
25 / 0.493536,
26 / 0.493646,
27 / 0.494699,
28 / 0.495214,
29 / 0.495897,
30 / 0.496142,
31 / 0.496837,
32 / 0.496912,
33 / 0.497383,
34 / 0.497632,
35 / 0.497831,
36 / 0.498070,
37 / 0.498290,
38 / 0.498398
}
\draw  ({(4*\x)/76}, \y) circle (0.0075);
\draw [thick] (0,0.5)--(2+.01,0.5);
\end{tikzpicture}
\caption[]{The plot 
  {\tikz[baseline=-0.5ex]\draw[fill=black,radius=1.5pt] (0,0) circle ;}
  for ${\rm Prob}(\chi(\mu)>0 \mid \chi(\mu)\neq 0)$ and 
  the plot
  {\tikz[baseline=-0.5ex]\draw[fill=white,radius=1.5pt] (0,0) circle ;}
  for ${\rm Prob}(\chi(\mu)<0 \mid \chi(\mu)\neq 0)$ where $1\leq n\leq 38$.%
}\label{Figure:PM:Plot}
\end{figure}

\subsubsection{}
Conjecture~\ref{General:Conjecture} implies that for any integer number $d$ one has
\begin{equation}
  {\rm Prob}(\,\chi(\mu)\equiv 0\,\, ({\rm mod}\ d)\,)\to 1\text{ as }n\to\infty.
\end{equation}
Figure~\ref{Figure:Plot} suggests that there is a sharper statement. See for example 
Figure~\ref{Figure:Comp}. 
 \begin{figure}[H]
  \centering
  \begin{tikzpicture}[scale=4]
    \draw [thick] (0,0-.004)--(0,1)--(2+.01,1)--(2+.01,0-.004)--cycle;
    \foreach \i in {2,4,6,8} \draw [thick] (0,\i/10) -- (.03,\i/10);
    \foreach \i in {2,4,6,8} \node[rotate=90] at (0-.05,\i/10) {\tiny\bf 0.\i};
    \node[rotate=90] at (0-.05,1) {\tiny\bf 1.0};
    \node[rotate=90] at (0-.05,0-.002) {\tiny\bf 0.0};
    \foreach \i in {10,20,...,76} \draw [thick] (2*\i/76,0-.004)--(2*\i/76,.03-.004);
    \foreach \i in {0,10,...,76} \node at (2*\i/76,0-.004-.05) {\tiny\bf \i};
  \foreach \x/\y in
  {
2/ 0.000000,
3/ 0.222222,
4/ 0.240000,
5/ 0.326530,
6/ 0.363636,
7/ 0.400000,
8/ 0.549587,
9/ 0.564445,
10/ 0.547619,
11/ 0.581633,
12/ 0.598414,
13/ 0.597392,
14/ 0.635720,
15/ 0.621578,
16/ 0.702611,
17/ 0.701470,
18/ 0.695632,
19/ 0.711920,
20/ 0.729827,
21/ 0.727850,
22/ 0.756931,
23/ 0.755566,
24/ 0.761004,
25/ 0.766302,
26/ 0.776543,
27/ 0.773131,
28/ 0.791316,
29/ 0.785669,
30/ 0.791326,
31/ 0.790687,
32/ 0.808667,
33/ 0.803730,
34/ 0.810735,
35/ 0.811763,
36/ 0.815064,
37/ 0.815565,
38/ 0.824422,
39/ 0.822188,
40/ 0.827024,
41/ 0.827150,
42/ 0.832165,
43/ 0.830679,
44/ 0.837467,
45/ 0.835640,
46/ 0.839228,
47/ 0.839611,
48/ 0.844193,
49/ 0.843245,
50/ 0.847102,
51/ 0.847389,
52/ 0.849612,
53/ 0.850410,
54/ 0.853485,
55/ 0.853151,
56/ 0.855968,
57/ 0.856350,
58/ 0.858603,
59/ 0.858999,
60/ 0.860868,
61/ 0.860969,
62/ 0.862982,
63/ 0.863306,
64/ 0.865487,
65/ 0.865570,
66/ 0.866821,
67/ 0.867150,
68/ 0.869040,
69/ 0.869127,
70/ 0.870139,
71/ 0.870653,
72/ 0.871719,
73/ 0.871960,
74/ 0.873174,
75/ 0.873273,
76/ 0.874223
  }
  \draw [fill] ({(2*\x)/76}, \y) circle (0.005);
\foreach \k in {2.0,2.2,...,75.8}
   \draw [thick] ({(2*\k)/76},{(2/3.1459)*rad(atan((\k/2)^(1/2)-1))}) --
     ({(2*(\k+0.2))/76},{(2/3.1459)*rad(atan(((\k+0.2)/2)^(1/2)-1))});
\end{tikzpicture}
\caption{The proportion of the character table of $S_n$ covered by
  even entries for $2\leq n\leq 76$ and the graph of 
$2\pi^{-1}\arctan(\sqrt{n/2}-1)$ for $2\leq n\leq 76$.
}\label{Figure:Comp}
\end{figure}
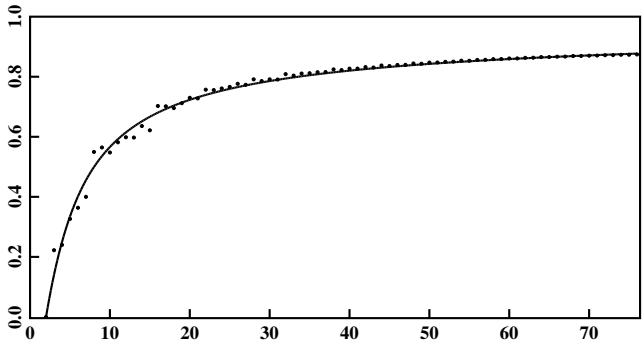
\newpage

\subsection{Tables}\label{Table:Section}\ 
\begin{table}[h]
  \caption{Number of even entries and number of odd entries
    in the character table of $S_n$ for $1\leq n\leq 76$.}\label{EO:Table}
  \scriptsize
  \centering
\begin{tabular}{lll}
  \toprule
  $n$ &  {no.~of evens} & {no.~of odds} \\
  \midrule
$1$ &  $0$ & $1$ \\
$2$ &  $0$ & $4$ \\
$3$ &  $2$ & $7$ \\
$4$ &  $6$ & $19$ \\
$5$ &  $16$ & $33$\\
$6$ &  $44$ & $77$\\
$7$ &  $90$ & $135$ \\
$8$ &  $266$ & $218$\\
$9$ &  $508$ & $392$\\
$10$ &  $966$ & $798$ \\
$11$ &  $1824$ & $1312$ \\
$12$ &  $3548$ & $2381$ \\
$13$ &  $6094$ & $4107$\\
$14$ &  $11586$ & $6639$\\
$15$ &  $19254$ & $11722$\\
$16$ &  $37492$ & $15869$ \\
$17$ &  $61876$ & $26333$ \\
$18$ &  $103110$ & $45115$\\
$19$ &  $170932$ & $69168$ \\
$20$ &  $286916$ & $106213$ \\
$21$ &  $456554$ & $170710$ \\
$22$ &  $759962$ & $244042$ \\
$23$ &  $1190034$ & $384991$\\
$24$ &  $1887766$ & $592859$ \\
$25$ &  $2937820$ & $895944$ \\
$26$ &  $4608084$ & $1326012$ \\
$27$ &  $7004646$ & $2055454$ \\
$28$ &  $10938762$ & $2884762$ \\
$29$ &  $16372732$ & $4466493$\\
$30$ &  $24851432$ & $6553384$ \\
$31$ &  $37014368$ & $9798596$ \\
$32$ &  $56368810$ & $13336991$ \\
$33$ &  $82688102$ & $20192347$ \\
$34$ &  $122855526$ & $28680574$ \\
$35$ &  $179808396$ & $41695293$ \\
$36$ &  $263406424$ & $59766105$ \\
$37$ &  $381814902$ & $86344867$ \\
$38$ &  $557951490$ & $118828735$ \\
\bottomrule
\end{tabular}\quad
\begin{tabular}{lll}
\toprule
$n$ &  {no.~of evens} & {no.~of odds} \\
\midrule
$39$ &  $799580980$ & $172923245$ \\
$40$ &  $1152977342$ & $241148902$ \\
$41$ &  $1644080076$ & $343563813$ \\
$42$ &  $2352923494$ & $474550782$ \\
$43$ &  $3324344208$ & $677609913$ \\
$44$ &  $4732761850$ & $918518775$ \\
$45$ &  $6639049122$ & $1305820834$\\
$46$ &  $9351080036$ & $1791411328$ \\
$47$ &  $13067332410$ & $2496228106$\\
$48$ &  $18309958344$ & $3379378185$\\
$49$ &  $25390864566$ & $4720061059$\\
$50$ &  $35331180090$ & $6377078986$\\
$51$ &  $48786461562$ & $8786181687$\\
$52$ &  $67367826002$ & $11924538919$\\
$53$ &  $92571070272$ & $16283394489$\\
$54$ &  $127268025536$ & $21847658489$ \\
$55$ &  $173744388742$ & $29905639434$\\
$56$ &  $237567368138$ & $39975105191$\\
$57$ &  $323002974632$ & $54182161084$\\
$58$ &  $439208932802$ & $72330715598$\\
$59$ &  $594363393060$ & $97561119340$\\
$60$ &  $804101537262$ & $129956924827$\\
$61$ &  $1082902860136$ & $174870604889$ \\
$62$ &  $1458789177232$ & $231616447104$\\
$63$ &  $1956705210484$ & $309822028517$\\
$64$ &  $2625259647972$ & $408015408928$\\
$65$ &  $3505898738012$ & $544490965352$\\
$66$ &  $4679753246976$ & $718991943424$\\
$67$ &  $6226771093726$ & $953962042995$\\
$68$ &  $8285512851154$ & $1248594579071$\\
$69$ &  $10979998587386$ & $1653369791639$\\
$70$ &  $14541318538948$ & $2170163830076$\\
$71$ &  $19209876952108$ & $2853857859917$\\
$72$ &  $25351409083192$ & $3730699401897$\\
$73$ &  $33363529811282$ & $4899218593439$\\
$74$ &  $43886589872232$ & $6374420377768$\\
$75$ &  $57554118617836$ & $8352091755860$\\
$76$ &  $75434276878574$ & $10852934727707$\\
  \bottomrule
\end{tabular}
\end{table}

\begin{table}[h]
  \caption{Number 
    of positive entries  
    and number of negative entries 
     in the character table of $S_n$ for ${1\leq n\leq 38}$.}\label{PMZ:Table}
  \scriptsize
  \centering
  \begin{tabular}{lll}
    \toprule
    $n$ & pos. & neg. \\
    \midrule
    1 & 1 & 0 \\
    2 & 3 & 1 \\
    3 & 6 & 2 \\
    4 & 14 & 7 \\
    5 & 26 & 13 \\
    6 & 58 & 34 \\
    7 & 98 & 72 \\
    8 & 194 & 137 \\
    9 & 344 & 249 \\
    10 & 652 & 524 \\
    11 & 1165 & 953 \\
    12 & 2020 & 1679 \\
    13 & 3552 & 3106 \\
    14 & 6077 & 5270 \\
    15 & 10362 & 9398 \\
    16 & 17080 & 15666 \\
    17 & 28570 & 26284 \\
    18 & 46836 & 43409 \\
    19 & 77045 & 72861 \\
        \bottomrule
  \end{tabular}\quad
  \begin{tabular}{lll}
    \toprule
    $n$ & pos. & neg. \\
    \midrule
    20 & 122013 & 115940 \\
    21 & 198461 & 189476 \\
    22 & 310602 & 297929 \\
    23 & 494008 & 476904 \\
    24 & 767237 & 743094 \\
    25 & 1205391 & 1174624 \\
    26 & 1828252 & 1782368 \\
    27 & 2846995 & 2787256 \\
    28 & 4277605 & 4196505 \\
    29 & 6520106 & 6413986 \\
    30 & 9795470 & 9645485 \\
    31 & 14738493 & 14553197 \\
    32 & 21750402 & 21483398 \\
    33 & 32582580 & 32243250 \\
    34 & 47614253 & 47165359 \\
    35 & 70213289 & 69606943 \\
    36 & 102477724 & 101689585 \\
    37 & 149340038 & 148321445 \\
    38 & 215267489 & 213892988 \\
    \bottomrule
  \end{tabular}
\end{table}
\begin{table}[h]
  \caption{Number of entries $\equiv 0\pmod d$
    in the character table of $S_n$ for $3\leq d\leq 7$ and $1\leq n\leq 19$.}\label{Other:Table}
  \scriptsize
  \centering
  \begin{tabular}{llllll}
    \toprule
    $n$ & $d=3$ & $d=4$ & $d=5$ & $d=6$ & $d=7$ \\
    \midrule
    1,2 & 0 & 0 & 0 & 0 & 0 \\
    3 & 1 & 1 & 1 & 1 & 1 \\
    4 & 6 & 4 & 4 & 4 & 4 \\
    5 & 11 & 12 & 12 & 11 & 10 \\
    6 & 39 & 30 & 35 & 29 & 29 \\
    7 & 73 & 61 & 64 & 59 & 63 \\
    8 & 181 & 187 & 178 & 163 & 168 \\
    9 & 426 & 368 & 336 & 352 & 339 \\
    10 & 803 & 681 & 726 & 643 & 660 \\
    11 & 1456 & 1272 & 1219 & 1188 & 1147 \\
    12 & 3138 & 2722 & 2668 & 2542 & 2503 \\
    13 & 5289 & 4532 & 4359 & 4135 & 3989 \\
    14 & 9980 & 8443 & 8332 & 8088 & 8031 \\
    15 & 16935 & 14067 & 14173 & 13363 & 13108 \\
    16 & 29669 & 27733 & 25351 & 25171 & 24066 \\
    17 & 49768 & 45156 & 42136 & 42202 & 39316 \\
    18 & 88645 & 77206 & 72601 & 73047 & 68206 \\
    19 & 139983 & 126447 & 115972 & 116635 & 108050 \\
    \bottomrule
  \end{tabular}
\end{table}
\section*{}
  \vspace{-.3in}
\end{document}